\documentclass[11pt,fleqn]{article}

\input{epsf.tex}
\usepackage{amsmath, amsthm, amssymb, amscd, amsfonts}
\usepackage{latexsym,verbatim, import, textcomp}
\usepackage{graphicx, epsfig, psfrag, wrapfig, epsf}
\usepackage{fullpage}
\usepackage{mathptmx}
\usepackage{url}
\numberwithin{equation}{section}

\theoremstyle{plain}
\newtheorem{theorem}{Theorem}[section]
\newtheorem{lemma}[theorem]{Lemma}

\newtheorem{conjecture}[theorem]{Conjecture}

\newtheorem{proposition}[theorem]{Proposition}

\newtheorem{claim}{Claim}

\theoremstyle{definition}
\newtheorem{definition}[theorem]{Definition}

\theoremstyle{remark}

\def\Mod{\mathrm{Mod}}

\def\cN{\mathcal N}
\def\C{\mathcal C}

\def\T{\tau}

\begin{document}

\title{Minimal pseudo-Anosov translation lengths on the complex of curves}
\author{Vaibhav Gadre \and Chia-Yen Tsai}

\maketitle
{\abstract \noindent We establish bounds on the minimal asymptotic pseudo-Anosov translation lengths on the complex of curves of orientable surfaces. In particular, for a closed surface with genus $g \geqslant 2$, we show that there are positive constants $a_1 < a_2$ such that the minimal translation length is bounded below and above by $a_1/ g^2$ and $a_2/g^2$.}

\section{Introduction}
Let $S_{g,n}$ be an orientable surface with genus $g$ and $n$ punctures. For simplicity, we shall drop the subscripts and denote it by $S$. The complex of curves $\C(S)$, is a locally infinite simplicial complex whose vertices are the isotopy classes of essential, non-peripheral, simple closed curves on $S$. A collection of vertices span a simplex if there are representatives of the curves that can be realized disjointly on the surface. Here, we will assume that the surface $S$ is non-sporadic i.e., the complexity $\xi(S)=3g-3+n \geqslant 2$. For sporadic surfaces, the complex of curves $\C(S)$, is either trivial or well-understood.

The mapping class group $\Mod(S)$, is the group of isotopy classes of diffeomorphisms of $S$. This group acts on $\C(S)$ in the obvious way. Thurston classified the elements of $\Mod(S)$ into three types: finite order, reducible or pseudo-Anosov. Given a mapping class $f \in \Mod(S)$, its asymptotic translation length on ${\C}(S)$ is defined to be
\begin{align*}
\ell_\C(f)= \liminf _{j\rightarrow \infty} \frac{d_\C(\alpha, f^j(\alpha))}{j}
\end{align*}
where $\alpha$ is a simple closed curve on $S$. The above limit remains unchanged when the numerator is changed by an additive constant. Hence, by the triangle inequality for $d_\C$, the quantity $\ell_\C(f)$ is independent of the choice of curve $\alpha$.

In \cite{MM}, Masur and Minsky proved that $f \in \Mod(S)$ is pseudo-Anosov if and only if $\ell_\C(f)> 0$. In \cite{Bd-tight}, Bowditch refined this, proving that the set of translation lengths of pseudo-Anosov elements is a closed discrete set in $\mathbb{R}^+$. In fact, Bowditch showed that the asymptotic translation lengths are rational with bounded denominator. We denote the minimal positive number in this set by
\begin{align*}
L_\C(\Mod(S))=\min \{\ell_\C(f) | f \in \Mod(S), \text{ pseudo-Anosov} \}
\end{align*}
For closed surfaces, Farb-Leininger-Margalit \cite{FLM} proved that when $g \geqslant 2$,
\begin{align*}
L_\C(\Mod(S))< \frac{4 \log(2+ \sqrt{3})}{g \log (g-\frac{1}{2})}.
\end{align*}
Here, we find a better upper bound for $L_\C(\Mod(S))$. Moreover, we also show that a lower bound of the same order holds. To be precise, we show:
\begin{theorem}\label{mainthm}
For closed surfaces with $g \geqslant 2$,
\begin{align*}
\frac{1}{162(2g-2)^2+ 6(2g-2)} <  L_\C(\Mod (S)) \leqslant \frac{4}{g^2+g-4}.
\end{align*}
\end{theorem}

The upper bound is established by bounding $\ell_\C(f)$ in examples of pseudo-Anosov maps and not expected to be sharp. The examples we use are a subset of those considered by Farb-Leininger-Margalit, but we obtain better bounds for $\ell_\C(f)$. The proof of the lower bound follows the same approach as Masur and Minsky \cite{MM}, but by keeping track of more information, we obtain sharper bounds. The additional information comes from the algorithm of Bestvina-Handel that constructs an invariant train track for a pseudo-Anosov map.

The lower bound in Theorem \ref{mainthm} is a part of Theorem \ref{lowerbd} in which we also prove a lower bound for punctured surfaces. To be precise, when $\xi(S)\geqslant 2$ and $n > 0$ we show that
\begin{align*}
\frac{1}{18(2g-2+n)^2+ 6(2g-2+n)} < L_\C(\Mod (S)).
\end{align*}
At the end of the paper, we also discuss upper bounds for some families of punctured surfaces.

\subsubsection*{Acknowledgements:} We thank Chris Leininger for suggesting the project, and numerous discussions during the course of it.

\section{The complex of curves}

One is mainly interested in the coarse geometry of $\C(S)$. The curve complex $\C(S)$ is quasi-isometric to its 1-skeleton equipped with the path metric. The 1-skeleton is a locally infinite graph.
In Proposition 4.6 of \cite{MM}, Masur and Minsky showed that:

\begin{proposition} \label{MM-trans}
For a non-sporadic surface $S$, there exists $c>0$ such that, for any pseudo-Anosov mapping class $f$ and any simple closed curve $\alpha$ in $\C(S)$
\[
d_\C(f^n(\alpha), \alpha) \geqslant c \vert n \vert
\]
for all $n \in {\mathbb Z}$.
\end{proposition}

In particular, the proposition shows that the curve complex $\C(S)$ has infinite diameter. In the same paper, Masur and Minsky went on to show that the curve complex is $\delta$-hyperbolic in the sense of Gromov. Proposition~\ref{MM-trans} implies that psuedo-Anosov mapping classes have ``north-south'' dynamics on $\C(S)$ i.e., they act as {\em hyperbolic} elements on $\C(S)$ and have an invariant quasi-axis.

A consequence of Proposition~\ref{MM-trans} is that for non-sporadic surfaces $S$, the minimal asymptotic translation length $L_\C(\Mod(S)) > 0$. In fact, Bowditch showed that the numbers $\ell_\C(f)$ are rational with uniformly bounded denominators \cite{Bd-tight}.

The following fact about the asymptotic lengths of iterates of $f$ is useful for proving bounds.
\begin{lemma}\label{power}
For all integers $m\geqslant1$, $\ell_\C (f^m)=m\ell_\C (f)$.
\end{lemma}

\begin{proof} From the definition of $\liminf$
\begin{eqnarray*}
\ell_\C(f) &=&  \liminf _{j\rightarrow \infty} \frac{d_\C(\alpha, f^j(\alpha))}{j} \leqslant \liminf _{j\rightarrow \infty} \frac{d_\C(\alpha, f^{jm}(\alpha))}{jm}\\
&=& \frac{1}{m}\liminf _{j\rightarrow \infty} \frac{d_\C(\alpha, f^{jm}(\alpha))}{j}=\frac{1}{m}\ell_\C(f^m)
\end{eqnarray*}
To get the reverse inequality, we use the triangle inequality, followed by the fact that $f$ is an isometry of $\C(S)$ i.e.,
\[
d_\C(\alpha, f^{jm}(\alpha)) \leqslant \sum_{i=1}^m d_\C(f^{(i-1)j}(\alpha), f^{ij}(\alpha)) = m d_\C(\alpha, f^j(\alpha))
\]
Hence $\ell_\C(f^m) \leqslant m \ell_\C(f)$ and we are done.
\end{proof}

\section{Train Tracks}\label{tt-construction}

For a detailed discussion of train tracks, see \cite{PeHa}.  We summarize the necessary definitions here.

A train track $\T$ on the surface is an embedded 1-dimensional CW complex with some additional structure.  The edges are called branches and the vertices are called switches.  The branches are smoothly embedded on the interiors, and there is a common point of tangency to all branches meeting at a switch.  This splits the set of branches incident on a switch into two disjoint subsets, which can be arbitrarily assigned as the incoming and outgoing edges at the switch. We assume that the valence of each switch is  at least 3.

A {\em train route} is a regular smooth path in $\T$. In particular, it traverses a switch only by passing from an incoming edge to an outgoing edge or vice versa. A train track $\sigma$ is {\em carried} by $\T$, denoted by $\sigma \prec \T$, if there is a homotopy of the identity map of the surface such that every train route in $\sigma$ is taken to a train route in $\T$. In particular, this means that $\sigma$ can be embedded in an $\epsilon$ neighborhood of $\T$. A simple closed curve is carried by a train track if it is homotopic to a closed train route.

An assignment of non-negative numbers, called weights, to the branches so that at every switch, the sum of the incoming weights equals the sum of the outgoing weights is called a {\em transverse measure} on the train track.  A closed train route induces a counting measure on $\T$.

Following Masur-Minsky \cite{MM}, we shall denote the set $P(\T)$ to be the polyhedron of transverse measures supported on $\T$ and let $int(P(\T)) \subset P(\T)$ be the set of transverse measures on $\T$ which induce positive weights on every branch of $\T$.  A simple closed curve $\alpha$ carried by $\T$ naturally induces a transverse measure supported on $\T$, so $\alpha \in P(\T)$.

A train track is called {\em large} if all the complementary regions are polygons or once-punctured polygons. A train track that has complementary regions ideal triangles or once-punctured monogons is called {\em maximal} or {\em complete}. It is maximal in the sense that it cannot be a sub-track of some other train track.

A train track $\T$ is called {\em recurrent} if there is a transverse measure which is positive on every branch of $\T$.  A train track $\T$ is {\em transversely recurrent} if given a branch of $\T$ there is a simple closed curve on $S$ that crosses the branch and intersects $\T$ transversely and efficiently i.e. the union of $\T$ and the simple closed curve has no complementary bigons. A train track that is both recurrent and transversely recurrent is {\em birecurrent}.

A train track $\sigma$ {\em fills} $\T$ if it is carried by $\T$ and $int(P(\sigma)) \subset int(P(\T))$. For recurrent train tracks, this means that every branch of $\T$ is traversed by some branch of $\sigma$.

For a large train track $\T$, a train track $\sigma$ is called a {\em diagonal extension} of $\T$ if $\T$ is sub-track of $\sigma$, and each branch in $\sigma \setminus \T$ has its endpoints terminate in the cusps of a complementary regions of $\T$.  Let $E(\T)$ denote the set of all recurrent diagonal extensions of $\T$. It is obvious that this is a finite set. Following Masur and Minsky, set
\[
PE(\T) = \bigcup_{\sigma \in E(\T)} P(\sigma)
\]
Further, let $int(PE(\T))$ be the set of measures in $PE(\T)$ that are positive on every branch of $\T$.

We begin with a preliminary lemma of Masur and Minsky which will be useful in Section~\ref{lbounds}; the proof of the lemma can be found in \cite{MM}.
\begin{lemma}[\cite{MM}]\label{carry}
For large recurrent train tracks $\sigma, \T$, if $\sigma$ is carried by $\T$ and fills $\T$, then any $\sigma' \in E(\sigma)$ is carried by some $\T' \in E(\T)$. In particular, there is the inclusion $PE(\sigma) \subset PE(\T)$.
\end{lemma}

\subsection*{The nesting lemma of Masur and Minsky:}

Given a set $A$ in $\C(S)$, let ${\cal N}_1(A)$ denote the 1-neighborhood of $A$ in $\C(S)$. In \cite{MM}, Masur and Minsky showed the following important result:

\begin{lemma}[Nesting lemma]\label{nesting}
Let $\T$ be a large birecurrent train track. Then
\[
{\cal N}_1(int(PE(\T)) \subset PE(\T).
\]
\end{lemma}
In other words, if $\alpha$ is a curve carried by a diagonal extension of $\T$ such that $\alpha$ passes through every branch of $\T$, and $\beta$ is a curve disjoint from $\alpha$ then $\beta$ is also carried by some diagonal extension of $\T$.

The original lemma in Masur and Minsky requires that $\T$ be birecurrent. We show below that the hypothesis of transverse recurrence can be dropped. The proof here was suggested to us by Chris Leininger.

\begin{proof}
Let $\T$ be a large recurrent train track and let $\alpha$ be a curve in $int(PE(\T))$. Let $\sigma$ be a diagonal extension of $\T$ carrying $\alpha$ such that $\alpha$ passes over every branch of $\sigma$ i.e., to get $\sigma$ we add to $\T$ only as many diagonals as necessary. Thus, $\alpha \in int(P(\sigma))$. We claim:

\begin{claim}
Let $\beta$ be a curve disjoint from $\alpha$. Then, $\beta \in PE(\sigma)$.
\end{claim}

\begin{proof}
For each branch $b$ of $\sigma$, denote by $\alpha(b)$ the weight assigned to $b$ by $\alpha$. Since $\alpha \in int(P(\sigma))$, the weights $\alpha(b) > 0$ for all $b$. 

Consider $\sigma$ as an abstract train track. To each branch $b$, assign a rectangle $R(b)$, of length 1 and width $\alpha(b)$. Foliate each rectangle by the product foliations i.e., by horizontal and vertical lines. The weights $\alpha(b)$ satisfy the switch conditions. So the rectangles glue along their widths in a consistent manner to give a neighborhood $\cN$ of $\sigma$. See Figure 1.

\begin{figure}[htb]
\begin{center}
\ \psfig{file=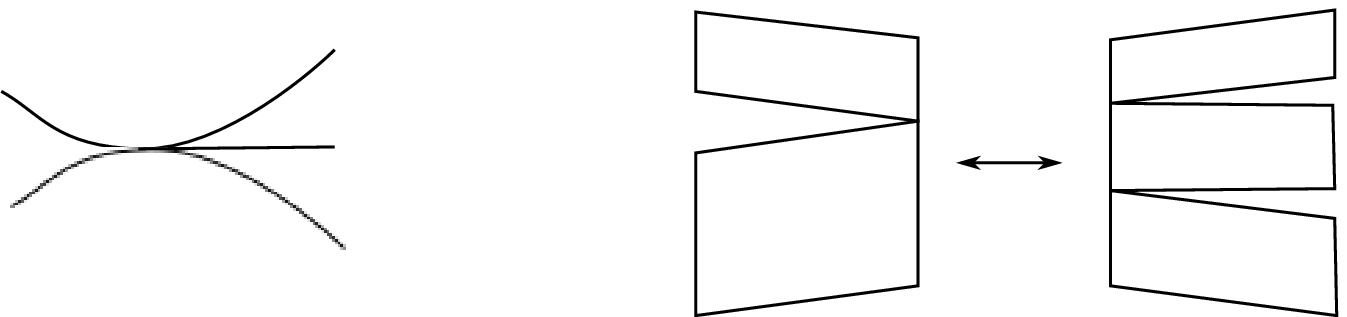, height=1truein, width=5truein} \label{glue-rect} \caption{Gluing rectangles}
\end{center}
\setlength{\unitlength}{1in}
\begin{picture}(0,0)(0,0)
\put(0.6,1.4){1}\put(0.6,0.9){2}\put(2.2,1.5){3}\put(2.2,1.15){4}\put(2.2,0.8){5}\put(2.9,1.5){$R(1)$}\put(2.9,0.9){$R(2)$}\put(5.9,1.5){$R(3)$}\put(5.9,1.2){$R(4)$}\put(5.9,0.8){$R(5)$}
\end{picture}
\end{figure}

The foliations also glue up to give a pair of singular foliations of $\cN$, which we continue to call horizontal and  vertical. The horizontal foliation is obtained from a cylinder neighborhood of $\alpha$, foliated by leaves parallel to $\alpha$, and with parts of its boundary glued together. In particular, we view $\alpha$ as a leaf in the horizontal foliation of $\cN$. The vertical foliations is by {\em ties} for $\sigma$. The components of the boundary $\partial \cN$, are each a finite union of arcs of singular leaves of the horizontal foliation, and correspond precisely to the complementary polygons of $\sigma$. In fact, $\cN$ admits an embedding into the surface $S$ as a neighborhood of $\sigma$.

The union of vertical sides of all rectangles is a union of leaves of the vertical foliation. Denote this union as $L$. Let $Z_1, \cdots, Z_u$ be the complementary polygons of $\sigma$. The key observation is that each $Z_i$ (that has say $k$ sides) is contained in a unique $2k$-gon $Y_i$ whose sides in a cyclic order, are alternatively arcs of $L$ and arcs of $\alpha$. See Figure 2. A side of $Z_i$ is a union of branches, and the rectangles corresponding to the branches are shown. The dotted lines in the union of these rectangles are the sides of the $2k$-gon that are arcs of $\alpha$. Each of these arcs is a piece of $\alpha$ first encountered as we move out from $Z_i$ across its sides. 

\begin{figure}[htb]
\begin{center}
\ \psfig{file=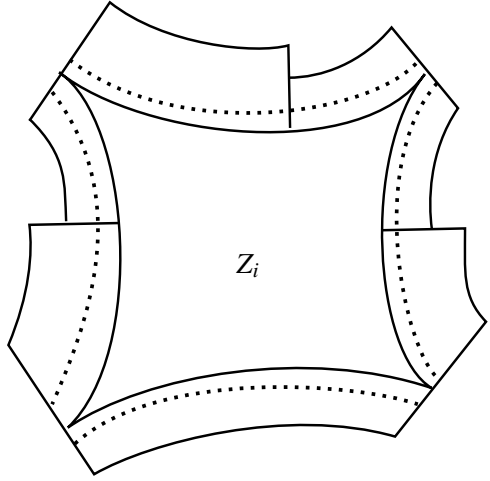, height=2.5truein, width=2.5truein} \label{2k-gon} \caption{The $2k$-gon}
\end{center}
\setlength{\unitlength}{1in}
\begin{picture}(0,0)(0,0)
\put(3.2,1.7){$Z_i$}
\end{picture}
\end{figure}

The surface $S$ decomposes into a union of the even-gons and a set of rectangles $X_1, \cdots, X_v$. For each $X_i$, a pair of opposite sides are arcs of $L$, and the other pair of opposite sides are arcs of $\alpha$. Also,  each $X_i$ is contained in an original rectangle $R(b_i)$, and that $R(b_i)$ contains no other $X_k$. 

Keeping $\beta$ disjoint from $\alpha$, we isotope $\beta$ to minimize the number of intersection points with $L$. By construction, each arc in $\beta \setminus \beta \cap L$ has to be contained entirely in either a single rectangle $X_i$ or a single even-gon $Y_j$, and must connect a $L$-side to another $L$-side. A arc inside $X_i$ connecting its $L$-sides traverses the branch $b_i$. If an arc in $Y_j$ connects consecutive $L$-sides in a cyclic order on the $L$-sides of $Y_j$, then it traverses a side of $Z_j$, which is a union of branches of $\sigma$. On the other hand, if an arc in $Y_j$ connects non-consecutive $L$-sides, then it traverses a diagonal of $Z_j$. 

It follows that $\beta$ is carried by a diagonal extension of $\sigma$, proving the claim.
\end{proof}

\noindent Diagonal extensions of $\sigma$ are also diagonal extensions of $\T$. So, the claim implies that $\beta \in PE(\T)$, finishing the proof of Lemma~\ref{nesting}.
\end{proof}

\section{The Bestvina-Handel algorithm}\label{BH}

\begin{definition}
Given a pseudo-Anosov mapping class $f\in \Mod(S)$, a train track $\T$ is an invariant train track of $f$ if $\T$ is large and recurrent, and $f(\T) \prec \T$.
\end{definition}

\subsubsection*{The Bestvina-Handel train track:}
The Bestvina-Handel algorithm takes as input a punctured surface with a pseudo-Anosov map $f$ and constructs an invariant train track $\T$ for $f$. The algorithm extends to closed surfaces as follows: Given a pseudo-Anosov map $f$ of a closed surface, a singularity of the stable foliation has a finite orbit under $f$. After puncturing the surface at these orbit points, the map $f$ restricts to a pseudo-Anosov map of the punctured surface. Running the Bestvina-Handel algorithm for the punctured surface yields a train track $\T$ that is also an invariant train track for the closed surface.
For details about the algorithm, we refer to \cite{BH}. Here, we present the features of the track $\T$ that we need in the proof of Theorem~\ref{lowerbd}.

\begin{enumerate}
\item The branches of $\T$ are essentially of two types: {\em real} and {\em infinitesimal}. The reason for this classification is that in passing from $\T$ to the associated Markov partition for the stable foliation, only the real branches correspond to rectangles. The algorithm can be carried out such that the total number of branches of $\T$ is bounded above by $ 9 \vert \chi(S) \vert - 3n$, and the number of real branches $r < 3 \vert \chi(S) \vert - 3$, where $\chi(S)$ is the Euler characteristic of $S$.

\item Along with the track $\T$, the algorithm gives a map $h: \T \to \T$ taking switches to switches, that is efficient in a certain sense. See Lemma 3.1.2 of \cite{BH}. There is a homotopy of $S$ sending $f(\T)$ into $\T$ such that the resulting map from $\T$ to itself is $h$. Since $h$ maps switches to switches, there is an unambiguously defined transition matrix $M$ with entries corresponding to ordered pairs of branches in $\T$ such that the entry corresponding to the pair $(b_1, b_2)$ counts the number of times $h(b_2)$ passes over $b_1$. Also, because switches are mapped to switches, the transition matrix for $f^k$ is $M^k$. Bestvina and Handel showed that the square submatrix $M_{\cal R}$ of $M$ obtained by restricting to the set of real branches ${\cal R}$ is irreducible, and in fact, an {\em integral Perron-Frobenius} matrix. We define a Perron-Frobenius matrix below:
\begin{definition}
A matrix $M$ is
\begin{enumerate}
\item irreducible if for any $(i,j)$, there exist a positive integer $s$, such that $M^s$ has a positive $(i,j)$-th entry.
\item non-negative if every entry of $M$ is non-negative.
\item Perron-Frobenius if it is irreducible and non-negative.
\end{enumerate}
\end{definition}
Additionally, for every infinitesimal branch in $\T$, there is a real branch such that some iterate of it passes over the infinitesimal branch.

A consequence of the above discussion is the lemma that follows: With $n$ as the number of punctures, set $c_0 = 162$ and $c_n = 18$ for all $n > 0$.
\begin{lemma}\label{p-matrix}
Given any real branch $\beta$, there is positive integer $k < c_n\chi(S)^2$ such that $f^k(\beta)$ passes over every branch of $\T$.
\end{lemma}
\begin{proof}
In \cite{BH}, Bestvina and Handel show that the transition matrix $M$ for $\T$ has the form:
\begin{align*}
M=\left( \begin{array}{cc}
 A & B\\
 0  & M_{\cal R}\\
 \end{array} \right),
\end{align*}
where $M_{\cal R}$ is an $r \times r$ integral Perron-Frobenius matrix.
By Perron-Frobenius theory \cite{Gm}, there exist a positive integer $q \leqslant  r$ such that $(M_{\cal R})^q$ has a positive diagonal entry.  From the form for the transition matrix $M$, the matrix $(M_{\cal R})^q$ equals the square submatrix $(M^q)_{\cal R}$ of $M^q$ obtained by restricting to ${\cal R}$, so $M_{\cal R}^q:=(M_{\cal R})^q=(M^q)_{\cal R}$.  Since $M^q$ is the transition matrix for the iterate $f^q$, the matrix $M^q_{\cal R}$ is still integral Perron-Frobenius matrix. By Proposition 2.4 in \cite{Ts} applied to $M^q_{\cal R}$, we know that $(M^q_{\cal R}) ^{2r}$ is a positive matrix.

If we set $p=2rq$, then for any real branch $\beta \in {\cal R}$, the path $f^p(\beta)$ crosses all real branches. This means that ${\cal R} \subseteq f({\cal R})$. Since every infinitesimal branch gets passed over by an iterate of some real branch, this inclusion is strict. Iterating, we get the sequence ${\cal R} \subset f({\cal R}) \subset f^2({\cal R}) \cdots$ where the inclusion remains strict as long as$f^{j+1}({\cal R})$ spreads over a larger set of infinitesimal branches than $f^j({\cal R})$. Let $i$ be the smallest positive integer such that the sequence stabilizes i.e., $f^i({\cal R}) = f^{i+1}({\cal R})$. Then $i$ is less than or equal to the number of infinitesimal branches and $f^i({\cal R}) = \T$. Since the total number of branches in $\T$ is bounded above by $9\vert \chi(S)\vert-3n$, the number of infinitesimal branches in $\T$ is bounded above by $9 \vert \chi(S) \vert-3n - r$. Hence, $i \leqslant 9 \vert \chi(S) \vert-3n - r$.

Set $k=p+i$. Then, for any real branch $\beta \in {\cal R}$, the path $f^k(\beta)$ crosses all branches of $\T$. It remains to give an upper bound for $k$ in terms of $\chi(S)$. Recall that $0<r=|{\cal R}| \leqslant 3|\chi(S)|-3$. For a surface $S$ with punctures i.e., $n>0$, we have the bound
\begin{align*}
k &= p+i= 2rq+i \leqslant 2r^2 + (- 9\chi -3n -r) \leqslant r(2r-1) - 9\chi -3n\\
&<(-3\chi - 3)(-6\chi - 7) -9 \chi = 18\chi^2 + 30\chi +21 < 18\chi^2
\end{align*}
where $\chi = \chi(S_{g,n})< 0$. For a closed surface $S_g$, the total number of branches, and hence $r+i$, is still bounded above by $-9\chi(S)$. So
\begin{align*}
k = p + i = 2rq+i \leqslant 2r^2+i < 2(r+i)^2 \leqslant 2(-9\chi)^2=162 \chi^2
\end{align*}
where $\chi=\chi(S_g)< 0$.
\end{proof}

\item We shall regard a cusp of a complementary region of $\T$ as {\em foldable}, if the branches $b_1$ and $b_2$ that flank it {\em fold} under some iterate i.e., there is some iterate such that the  paths $h^j(b_1)$ and $h^j(b_2)$ starting from the same initial switch pass over the same initial branch $b$. By Property (I2) in Section 4 of \cite{BH}, the algorithm is carried out such that the branch $b$ that they fold over is always real.

Let $\sigma \in E(\T)$, and let $\gamma$ be a simple closed curve carried by  $\sigma$. We have the following lemma:
\begin{lemma}\label{inc-fold}
If $\gamma$ does not pass over any real branch of $\T$, then $\gamma$ is incident on a foldable cusp.
\end{lemma}

\begin{proof}
Suppose that $\gamma$ does not pass over any real branch of $\T$ and none of the cusps it passes through are foldable. Any iterate of $\gamma$ must also have the same properties. But then, the iterates cannot converge to the stable foliation of $f$, giving a contradiction.
\end{proof}

 \end{enumerate}

\section{Lower bounds}\label{lbounds}

\begin{theorem}\label{lowerbd}
When the complexity $\xi(S)\geqslant 2$,
\begin{align*}
L_\C  (\Mod (S)) > \frac{1}{c_n\chi(S)^2+ 6 \vert \chi(S) \vert},
\end{align*}
where, as in Lemma \ref{p-matrix}, the constants are $c_0=162$ and $c_n=18$ for $n \geqslant 1$.
\end{theorem}

\begin{proof}
The idea is to combine Lemma~\ref{p-matrix} with the proof of Proposition 4.6 in \cite{MM} to obtain a better lower bound.

For any pseudo-Anosov $f \in \Mod(S)$, let $\T$ be the invariant train track of $f$ constructed by Bestvina and Handel. We shall show that after at most $(6 \vert \chi(S) \vert + k) $ iterates of $f$, where $k$ is the number of iterates in Lemma \ref{p-matrix}, we get the nesting behavior in the proof of Proposition 4.6 in \cite{MM}. This gives the lower bound for $L_\C(\Mod(S))$ as stated.

For the track $\T$, let ${\cal B}_\T$ be the set of the branches, and $\vert {\cal B}_{\T} \vert$ its cardinality.  Let $\sigma \in E(\T)$. By Lemma \ref{carry}, the image $f(\sigma)$ is carried by some diagonal extension $\sigma' \in E(\T)$. It follows immediately that $f$ sends switches of $\sigma$ to switches of $\sigma'$. Hence, for each such pair $(\sigma, \sigma')$, the transition matrix $M_{\sigma,\sigma'}: \mathbb{R}^{{\cal B}_{\sigma}} \rightarrow \mathbb{R}^{{\cal B}_{\sigma'}} $ associated to $f$ is unambiguously defined. Without loss generality, we may assume that the last $| {\cal B}_\T |$ coordinates correspond to ${\cal B}_\T$. Then, the matrix $M_{\sigma,\sigma'}$ has the form
\begin{align*}
M_{\sigma,\sigma'}=
\left( \begin{array}{c|c}
 * &*\\  \hline
 0 & M\\
  \end{array} \right),
\end{align*}
where $M$ is the transition matrix for $\T$ described in Lemma \ref{p-matrix}.
For any diagonal extension $\sigma_0 \in E(\T)$, for any $m>0$, we can construct a sequence of train tracks $\sigma_1,\sigma_2,\cdots,\sigma_m$ in $E(\T)$ such that
\begin{align*}
f(\sigma_0) \prec \sigma_1, f(\sigma_1) \prec \sigma_2, \cdots, f(\sigma_{m-1}) \prec \sigma_m,
\end{align*}
hence
\begin{align*}
f^m(\sigma_0)\prec \sigma_m.
\end{align*}
Let $M_{\sigma_j,\sigma_{j+1}}: \mathbb{R}^{{\cal B}_{\sigma_j}} \rightarrow \mathbb{R}^{{\cal B}_{\sigma_{j+1}}} $ be the transition matrices of $f$ associated to $f(\sigma_j) \prec \sigma_{j+1}$ in the sequence, and let $M_{\sigma_{0},\sigma_m}: \mathbb{R}^{{\cal B}_{\sigma_m}} \rightarrow \mathbb{R}^{{\cal B}_{\sigma_m}} $ be the transition matrix associated to $f^m$. Because $f$ maps switches to switches, the matrices satisfy
\begin{align*}
M_{\sigma_{0},\sigma_m} &= M_{\sigma_{m-1},\sigma_m} \times M_{\sigma_{m-2},\sigma_{m-1}} \times \cdots \times M_{\sigma_{1},\sigma_{2}} \times M_{\sigma_{0},\sigma_{1}} \\
&= \left( \begin{array}{c|c}
 * &*\\  \hline
 0 & M\\
  \end{array} \right) \times
  \left( \begin{array}{c|c}
 * &*\\  \hline
 0 & M\\
  \end{array} \right) \times \cdots \times
  \left( \begin{array}{c|c}
 * &*\\  \hline
 0 & M\\
  \end{array} \right) \times
  \left( \begin{array}{c|c}
 * &*\\  \hline
 0 & M\\
  \end{array} \right) \\
 &= \left( \begin{array}{c|c}
 * &*\\  \hline
 0 & M^m\\
  \end{array} \right).
\end{align*}
We now use Lemma~\ref{p-matrix} and Lemma~\ref{inc-fold} to prove the following lemma:

\begin{lemma}\label{hooked}
For any $\mu \in P(\sigma_0)$, there exists some positive integer $m$ such that $k \leqslant m \leqslant 6 \vert \chi(S) \vert+ k$, where $k$ is the number of iterates in Lemma~\ref{p-matrix}, the measure $f^m(\mu) \in P(\sigma_m)$ is positive on every branch in ${\cal B}_\T$, that is $f^m(\mu) \in int(PE(\T))$.
\end{lemma}
\begin{proof}
We consider the simplest case first:

\vskip 10pt

\noindent {\bf Case 1:} Suppose $\mu$ is positive on some real branch $\beta$ in ${\cal B}_\T$. By Lemma~\ref{p-matrix}, the transition matrix with respect to $\T$, for $f^k$ has the form:
\begin{align*}
M^k=\left( \begin{array}{cc}
 *  & \widehat{B}\\
  0  & M^k_{\cal R}\\
\end{array} \right),
\end{align*}
where $\widehat{B}$ and $M^k_{\cal R}$ are positive matrices. In particular, the image path $f^k(\beta)$ passes over every branch in ${\cal B}_\T$. Hence, the measure $f^k(\mu) =  M_{\sigma_{0},\sigma_k}(\mu)$ in $P(\sigma_k)$ is positive on every branch in ${\cal B}_\T$. The same reasoning applied to all integers $m \geqslant k$ implies that the measure $f^m(\mu) = M_{\sigma_0, \sigma_m}(\mu)$ is positive on every branch of ${\cal B}_\T$, finishing the proof of Lemma~\ref{hooked} in this case.

\vskip 10pt

\noindent{\bf Case 2:} Suppose $\mu$ is not positive on any real branch. We shall show that in $j \leqslant 6 \chi(S)$ iterates the measure $f^j(\mu) = M_{\sigma_0, \sigma_j}(\mu)$ is positive on some real branch, reducing us to Case 1. This is done in two steps: In Step 1, we show that $Supp(\mu)$ contains a diagonal $d$ that is incident on a foldable cusp $c$. In Step 2, we show that the branches $b_1$ and $b_2$ that flank $c$, fold over a real branch $b$ in $j \leqslant 6\chi(S)$ iterates. Then $f^j(d)$ also passes over $b$ from which it follows that $f^j(\mu)$ assigns positive weight to $b$.

\vskip 10pt

\noindent {\it Step 1:}
Suppose $\mu$ is positive on some simple closed curve $\gamma$ carried by $\sigma_0$. By Lemma~\ref{inc-fold}, the curve $\gamma$ must be incident on a foldable cusp. Hence, $Supp(\mu)$ contains a diagonal $d$ that is incident on a foldable cusp $c$.

\vskip 10pt

\noindent {\it Step 2:} Let $b_1$ and $b_2$ be the branches that flank $c$. Let $j$ be the smallest iterate in which $b_1$ and $b_2$ fold. By Part (3) of Section~\ref{BH}, the branch $b$ that they fold over is real. We claim that $j \leqslant 6 \vert \chi(S) \vert$.  By an Euler characteristic calculation, the total number of cusps is at most $6 \vert \chi(S) \vert$. If $b_1$ and $b_2$ do not fold within $6 \vert \chi(S) \vert$ iterates, then there is a foldable cusp $c'$ such that $f^a(c') = c'$ for some iterate $f^a$. But then $f^{ma}(c') = c'$ for all positive integers $m$. Thus, $c'$ never gets folded giving a contradiction. This proves the claim.

\vskip 10pt

\noindent Finally, combining this with Case 1, we conclude that for $m = j +k  \leqslant 6 \vert \chi(S) \vert + k$, the measure $f^m(\mu)$ is positive on every branch in ${\cal B}_\T$ finishing the proof of Lemma~\ref{hooked}.
\end{proof}
\noindent Finally, Lemma~\ref{hooked} implies that for any $\sigma_0 \in E(\T)$, and for any $\mu \in P(\sigma_0)$,
\begin{align*}
f^w(\mu) \in int(PE(\T)),
\end{align*}
where $w = 6 \vert \chi(S)\vert + k$. Hence,
\begin{align}\label{nest}
f^w (PE(\T)) \subset int(PE(\T)).
\end{align}
Now set $\T_1=\T$, and for each positive integer $i \geqslant 1$, let $\T_i=f^{i w}(\T)$. The inclusion \eqref{nest} implies $PE(\T_{i+1}) \subset int(PE(\T_i))$. Choose $\beta \in \C(S) \setminus PE(\T_1)$ such that $f^w(\beta) \in PE(\T_1)$. By applying Lemma~\ref{nesting}, we have $d_\C(f^{iw}(\beta),\beta) \geqslant i$. Hence
\begin{align*}
\ell_\C(f^w)=\liminf_{ i \rightarrow \infty}\frac{d_\C(f^{iw}(\beta),\beta)}{i} \geqslant \liminf_{i \rightarrow \infty}\frac{i}{i} =1.
\end{align*}
By Lemma \ref{power}, we have $\ell_\C(f^w)=w \ell_\C(f)$. So
\begin{align*}
\ell_\C(f) \geqslant \frac{1}{w} > \frac{1}{c_n \chi(S)^2 + 6 \vert \chi(S) \vert},
\end{align*}
where $c_0=162$ and $c_n=18$ for $n\geqslant 1$.
\end{proof}

\section{Upper bound}
Next, for a closed surface $S$, we prove an upper bound for $L_\C (\Mod (S))$ of the same order.
\begin{theorem}\label{upperbd}
For a closed surface of genus $g \geqslant 2$,
\begin{align*}
L_\C (\Mod (S))\leqslant \frac{4}{g^2+g-4}.
\end{align*}
\end{theorem}

\begin{proof}
It is sufficient to find a pseudo-Anosov mapping class $f$ such that $\ell_\C(f) \leqslant \frac{4}{g^2+g-2}$. We show this for the pseudo-Anosov map of a closed surface of genus $g$ constructed by Penner in \cite{Pe}.  The Penner example is as follows: For the closed surface of genus $g$ in  Figure~\ref{pe-ex}, let $f =\rho T_{c_{1}} T_{b_{1}}^{-1}  T_{a_{1}}$, where $T_{a_{1}}$ is a positive Dehn twist along $a_1$, $\rho(a_i)=a_{i-1}$, for $i=2,\cdots,g$ and $\rho(a_1)=a_g$ and similarly for the $b_i$'s and $c_i$'s.
\begin{figure}[htb]
\begin{center}
\psfrag{a1}{$a_1$}
\psfrag{a2}{$a_2$}
\psfrag{a3}{$a_3$}
\psfrag{a4}{$a_4$}
\psfrag{a5}{$a_5$}
\psfrag{ag}{$a_{g}$}
\psfrag{ag-1}{$a_{g-1}$}
\psfrag{ag-2}{$a_{g-2}$}
\psfrag{ag-3}{$a_{g-3}$}
\psfrag{ag-4}{$a_{g-4}$}
\psfrag{ag-5}{$a_{g-5}$}
\psfrag{b1}{$b_1$}
\psfrag{b2}{$b_2$}
\psfrag{b3}{$b_3$}
\psfrag{b4}{$b_4$}
\psfrag{b5}{$b_5$}
\psfrag{bg}{$b_{g}$}
\psfrag{bg-1}{$b_{g-1}$}
\psfrag{bg-2}{$b_{g-2}$}
\psfrag{bg-3}{$b_{g-3}$}
\psfrag{bg-4}{$b_{g-4}$}
\psfrag{bg-5}{$b_{g-5}$}
\psfrag{c1}{$c_1$}
\psfrag{c2}{$c_2$}
\psfrag{c3}{$c_3$}
\psfrag{c4}{$c_4$}
\psfrag{c5}{$c_5$}
\psfrag{cg}{$c_{g}$}
\psfrag{cg-1}{$c_{g-1}$}
\psfrag{cg-2}{$c_{g-2}$}
\psfrag{cg-3}{$c_{g-3}$}
\psfrag{cg-4}{$c_{g-4}$}
\includegraphics[width=0.75\textwidth]{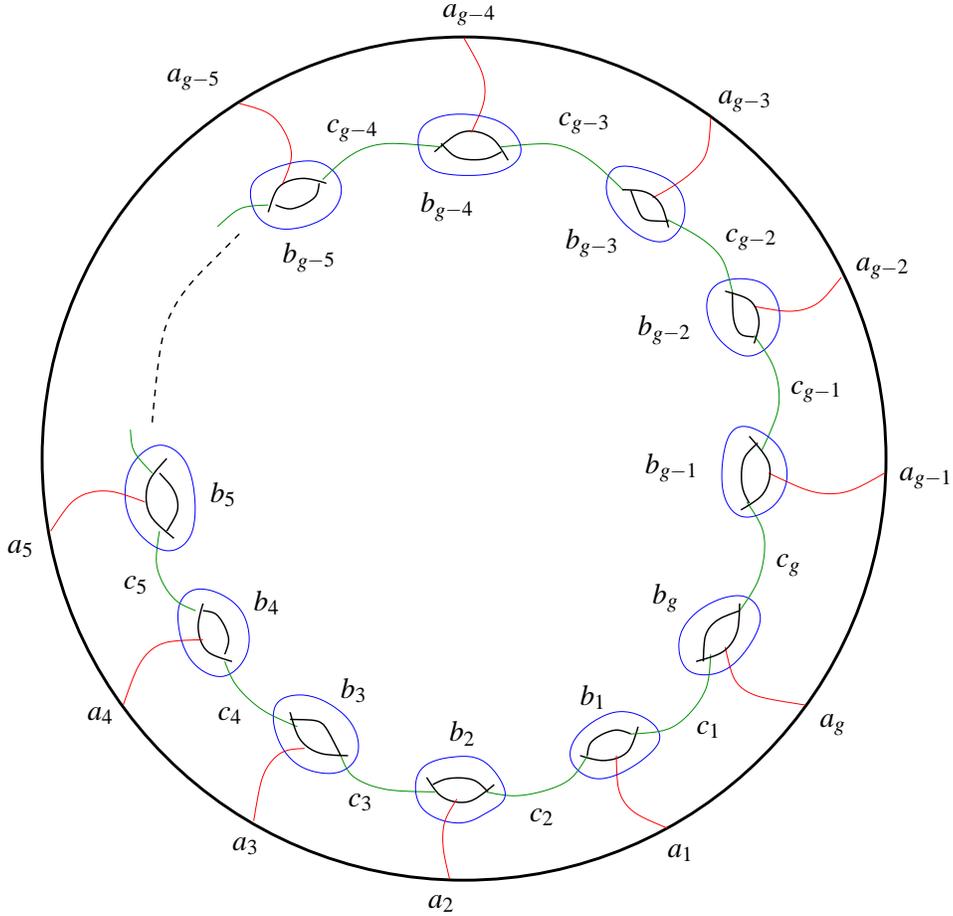}
\caption{$f=\rho  T_{c_{1}} T_{b_{1}}^{-1}  T_{a_{1}} \in \Mod (S)$.}\label{pe-ex}
\end{center}
\end{figure}

\noindent Since $\ell_\C(f)$ is independent of the initial choice of curve to apply iterations to, we choose the curve $a_g$ and show that for some $k \geqslant \frac{g^2+g-4}{2}$,
\begin{align}\label{dist-estimate}
d_\C (f^k(a_g),a_g)\leqslant 2.
\end{align}
By the triangle inequality,
\begin{align*}
\ell_\C(f^k)=\liminf_{j\rightarrow \infty}\frac{d_\C (f^{jk}(a_g),a_g)}{j} \leqslant \liminf_{j\rightarrow \infty}\frac{2j}{j}=2,
\end{align*}
and by Lemma~\ref{power},
\begin{align*}
\ell_\C(f)\leqslant \frac{2}{k} \leqslant \frac{4}{g^2+g-4}.
\end{align*}

\noindent For a sequence of curves $\alpha_r \in \{ a_i,b_i,c_i \}_{i=1}^g$ such that $\alpha_1 \cup \cdots \cup \alpha_k$ is connected, we denote the regular neighborhood of the union $\alpha_1 \cup \cdots \cup \alpha_k$ by ${\cal N}(\alpha_1, \cdots, \alpha_k)$. For $g=2$, it is easy to see
\begin{align*}
f(a_2)=a_1 \text{, } f^2(a_2) \subset {\cal N}(a_2b_2c_2).
\end{align*}
Thus $f^2(a_2)$ and $a_1$ are disjoint. Since $a_1$ and $a_2$ are disjoint, by triangle inequality $d_\C (f^2(a_2),a_2)\leqslant 2$ and $2 \geqslant \frac{2^2+2-4}{2}$, and we are done.

To show \eqref{dist-estimate} in general, the key idea is as follows: Suppose that $f^j$ is the smallest iterate in which $f^j(a_g)$ is spread over $k$ ``holes''. Then it takes waiting time $(g+1)$ for the images to sweep over $(k+2)$ holes. In other words, $f^{j+g+1}$ is the smallest iterate in which the image of $a_g$ sweeps over $(k+2)$ holes.  To be precise, among the neighborhoods defined above, we keep track of which is the ``smallest'' one containing the image of $a_g$.

In first $(g-1)$ iterates $a_g$ gets rotated till it becomes $a_1$ i.e., $f^{g-1}(a_g) = a_1$.  In two iterates that follow:
\begin{align*}
f^{g}(a_g) \subset {\cal N}(a_g b_g c_g) \text{ ,  }
f^{g+1}(a_g) \subset {\cal N}(c_g a_{g-1} b_{g-1}c_{g-1}).
\end{align*}
In the same manner, continuing the iterations, notice that:
\begin{align*}
f^{2(g+1)}(a_g) &\subset
{\cal N}\left(c_g a_{g-1} b_{g-1}c_{g-1}\cdots a_{g-3}b_{g-3}c_{g-3}\right).\\
f^{3(g+1)}(a_g) &\subset
{\cal N}\left(c_ga_{g-1}b_{g-1}c_{g-1}\cdots a_{g-5}b_{g-5}c_{g-5} \right).
\end{align*}
We observe that after each $f^{g+1}$ iterates the subscript for $c$ rightmost inside ${\cal N}$, decreases by 2. In other words, it requires $(g+1)$ iterates to increase by 2, the ``complexity'' of the image of $a_g$. Here, we simplify notation as follows:
\begin{align*}
{\cal N}(c_g * c_{g-i}) := {\cal N}(c_g a_{g-1} b_{g-1} \cdots  a_{g-i}b_{g-i} c_{g-i})
\end{align*}
Then, we have
\begin{align*}
f^{g+1}(a_g)&\subset {\cal N} (c_g * c_{g-1}),\\
f^{2(g+1)}(a_g)&\subset {\cal N}(c_g * c_{g-3})\\
f^{3(g+1)}(a_g) &\subset {\cal N}(c_g * c_{g-5})\\
\vdots \\
f^{\lfloor \frac{g-1}{2} \rfloor(g+1)}(a_g) &\subset {\cal N}\left( c_g \ast c_{g-2\lfloor \frac{g-1}{2} \rfloor+1} \right),\\
f^{g-1} \left( f^{\lfloor \frac{g-1}{2} \rfloor(g+1)}(a_g) \right) &\subset {\cal N}\left( c_1 a_g b_g c_g \ast c_{g-2\lfloor \frac{g-1}{2} \rfloor+1}\right).
\end{align*}
From the above inclusions, we note that $f^{g-1}\left( f^{\lfloor \frac{g-1}{2} \rfloor(g+1)}(a_g) \right)$ is disjoint from $a_1$, and of course $a_1$ and $a_g$ are disjoint initially. Hence
\begin{align*}
d_\C\left(a_g, f^k(a_g)\right) \leqslant 2,
\end{align*}
where
\[
k=(g-1)+  \lfloor \frac{g-1}{2} \rfloor(g+1) \geqslant \frac{2(g-1)+(g-2)(g+1)}{2} = \frac{g^2+g-4}{2}.
\]
\end{proof}

More generally, as described in Appendix 5.2 of \cite{Ts}, a method similar to \cite{Pe} constructs pseudo-Anosov homeomorphisms of certain punctured surfaces from pseudo-Anosov homeomorphisms of closed surfaces. We start with the Penner pseudo-Anosov map $f$ of the closed surface $S_g$. We add in punctures in some or all of the complementary regions according to the criteria of Theorem 3.1 in \cite{Pe-const}. Then, the restriction of $f$ is a pseudo-Anosov on the punctured surface. A proof similar to Theorem~\ref{upperbd} provides upper bounds on $L_{\C}(\Mod(S_{g,n}))$ of the order $1/\chi(S_{g,n})^2$. We list the cases in which we get $1/\chi(S_{g,n})^2$ type upper bounds:
\begin{enumerate}
\item For punctured tori with $n$ even: we use the example in Appendix 5.1 of \cite{Ts}.
\item For $g\geqslant 5$ and $n=g-1$ or $2g-2$: we use Example 1 in Appendix 5.2  of \cite{Ts} .
\item For $g\geqslant 3$ and $n\leqslant 4$: we use Example 2 in Appendix 5.2  of \cite{Ts} .
\item For $g\geqslant 2$ and $n=1$, $2$, $g$, $g+1$ or $g+2$: We use Penner's example in Theorem~\ref{upperbd}, puncturing the surface at the appropriate points.
\end{enumerate}

In some cases, the upper bound can be of the order of $1/\vert \chi(S_{g,n})\vert$. For example, when $g= 2$ and $n$ is varying, the example in Section 4 of \cite{Ts} gives the bound
\begin{align*}
L_\C  (\Mod (S_{g,n})) \leqslant \frac{20}{n-4},
\end{align*}
for all $n \geqslant 4$. We propose the following conjecture:
\begin{conjecture}
For fixed $g\geqslant 2$ and $n$ varying, $L_\C  (\Mod (S_{g,n}))$ is of the order of $\frac{1}{|\chi(S_{g,n})|}$ as $n \to \infty$.
\end{conjecture}

\bibliographystyle{amsalpha}

\bibliography{bib}

\begin{thebibliography}{FLM08}

\bibitem[BH95]{BH}
M.~Bestvina and M.~Handel.
\newblock Train-tracks for surface homeomorphisms.
\newblock {\em Topology}, 34(1):109--140, 1995.

\bibitem[Bow08]{Bd-tight}
Brian~H. Bowditch.
\newblock Tight geodesics in the curve complex.
\newblock {\em Invent. Math.}, 171(2):281--300, 2008.

\bibitem[FLM08]{FLM}
Benson Farb, Christopher~J. Leininger, and Dan Margalit.
\newblock The lower central series and pseudo-{A}nosov dilatations.
\newblock {\em Amer. J. Math.}, 130(3):799--827, 2008.

\bibitem[Gan59]{Gm}
F.~R. Gantmacher.
\newblock {\em The theory of matrices. {V}ols. 1, 2}.
\newblock Translated by K. A. Hirsch. Chelsea Publishing Co., New York, 1959.

\bibitem[MM99]{MM}
Howard~A. Masur and Yair~N. Minsky.
\newblock Geometry of the complex of curves. {I}. {H}yperbolicity.
\newblock {\em Invent. Math.}, 138(1):103--149, 1999.

\bibitem[Pen88]{Pe-const}
Robert~C. Penner.
\newblock A construction of pseudo-{A}nosov homeomorphisms.
\newblock {\em Trans. Amer. Math. Soc.}, 310(1):179--197, 1988.

\bibitem[Pen91]{Pe}
R.~C. Penner.
\newblock Bounds on least dilatations.
\newblock {\em Proc. Amer. Math. Soc.}, 113(2):443--450, 1991.

\bibitem[PH92]{PeHa}
R.~C. Penner and J.~L. Harer.
\newblock {\em Combinatorics of train tracks}, volume 125 of {\em Annals of
  Mathematics Studies}.
\newblock Princeton University Press, Princeton, NJ, 1992.

\bibitem[Tsa09]{Ts}
Chia-Yen Tsai.
\newblock The asymptotic behavior of least pseudo-anosov dilatations.
\newblock {\em Geom. Topol.}, 13(4):2253--2278, 2009.

\end{thebibliography}


\bigskip

\noindent Department of Mathematics, Harvard University, Cambridge, MA 02138. \newline
\noindent \texttt{vaibhav@math.harvard.edu}

\vskip15pt 

\noindent Department of Mathematics, University of Illinois, Urbana-Champaign, IL 61801. \newline \noindent \texttt{cyt1230@gmail.com}

\end{document}